\documentclass[a4paper,12pt,reqno,oneside]{amsart}
\usepackage[utf8]{inputenc}
\usepackage[finnish,english]{babel}
\usepackage{amsthm,amssymb,mathtools,hyperref,amsaddr}

\title[X-ray transform in non-smooth geometry]{Microlocal analysis of the X-ray transform in non-smooth geometry}
\author{Joonas Ilmavirta}
\email{joonas.ilmavirta@jyu.fi}
\address{Department of Mathematics and Statistics, University of Jyv\"askyl\"a}
\author{Antti Kykk\"anen}
\email{antti.k.kykkanen@jyu.fi}
\address{Department of Mathematics and Statistics, University of Jyv\"askyl\"a}
\author{Kelvin Lam}
\email{klam0008@uw.edu}
\address{Department of Mathematics, University of Washington}
\date{\today}

\newtheorem{theorem}{Theorem}
\newtheorem{proposition}[theorem]{Proposition}
\newtheorem{lemma}[theorem]{Lemma}
\newtheorem{corollary}[theorem]{Corollary}

\theoremstyle{definition}
\newtheorem{definition}[theorem]{Definition}
\newtheorem{remark}[theorem]{Remark}

\DeclareMathOperator{\supp}{supp}
\DeclareMathOperator{\Lip}{Lip}
\DeclareMathOperator{\Op}{Op}
\DeclareMathOperator{\Id}{Id}

\newcommand{\Vol}{\mathrm{Vol}}

\newcommand*{\N}{\mathbb{N}}

\newcommand*{\R}{\mathbb{R}}

\newcommand{\der}{\mathrm{d}}
\newcommand{\eps}{\varepsilon}

\newcommand{\abs}[1]{\left\vert#1\right\vert}

\newcommand{\norm}[1]{\left\lVert#1\right\rVert}

\newcommand{\doo}[1]{\partial_{\mathrm{#1}}}

\newcommand{\antti}[1]{}
\newcommand{\kelvin}[1]{}
\newcommand{\joonas}[1]{}

\begin{document}

\maketitle

\begin{abstract}
We prove that the geodesic X-ray transform is injective on $L^2$ when the Riemannian metric is simple but the metric tensor is only finitely differentiable. The number of derivatives needed depends explicitly on dimension, and in dimension 2 we assume $g\in C^{10}$. Our proof is based on microlocal analysis of the normal operator: we establish ellipticity and a smoothing property in a suitable sense and then use a recent injectivity result on Lipschitz functions. When the metric tensor is $C^k$, the Schwartz kernel is not smooth but $C^{k-2}$ off the diagonal, which makes standard smooth microlocal analysis inapplicable.
\end{abstract}

\section{Introduction}

We show that on a simple Riemannian manifold $(M,g)$ where $g\in C^k$ for a finite and explicit~$k$ the geodesic X-ray transform is injective on~$L^2$ (Theorem~\ref{thm:xrt-injective}).
We do this using a typical two-step approach, first showing that a function in the kernel of the transform is smoother than assumed a priori and then showing that injectivity holds for smooth functions.
Both of the two steps of the proof have to be adapted to low regularity.
The ``smooth'' injectivity (on Lipschitz functions) was established in~\cite{IlmKyk2021}, so it remains to prove that a function in the kernel of the X-ray transform has to be Lipschitz.

This regularity result (Theorem~\ref{thm:elliptic-regularity}) is based on microlocal analysis of the normal operator.
This normal operator is not a pseudodifferential operator in the usual sense because the ``smooth'' off-diagonal part of the Schwartz kernel is only~$C^{k-2}$.
Also, when the metric tensor is not infinitely differentiable, the Sobolev scale of~$H^s$ spaces only makes sense for a bounded range of indices~$s$ in both the positive and the negative direction.
These two issues mean that the concepts of ellipticity, smoothing, and a parametrix need careful treatment.

\subsection{Main results}

We consider two operators: The X-ray transform~$I$ and its normal operator~$N$.
These are defined separately, and we only prove that $N=I^*I$ when acting on $L^2$ functions.
Precise definitions of the operators and spaces we employ are given in section~\ref{sec:preliminaries} below.

We prove two main theorems. Theorem~\ref{thm:elliptic-regularity} concerns functions in the kernel of the operator~$N$ and proves that they have, a priori, improved regularity. Theorem~\ref{thm:xrt-injective} can be compared to a recent result in~\cite{IlmKyk2021}. We prove that the X-ray transform is injective on~$L^2(M)$ while requiring more metric regularity whereas~\cite[Theorem 1]{IlmKyk2021} proves that the X-ray transform is injective only on Lipschitz functions.

\begin{theorem}
\label{thm:elliptic-regularity}
Let $(M,g)$ be a simple manifold, $n \coloneqq \dim M \ge 2$ and $g \in C^k(M)$ for some $k \ge 7 + \frac{n}{2}$.
Then if $f \in H^s_c(M)$ for some $s > -k+6+\frac{n}{2}$ and $Nf = 0$, we have $f \in H^r_c(M)$ for all $s < r < k-6-\frac{n}{2}$.
\end{theorem}

Theorem~\ref{thm:elliptic-regularity} can be applied to geodesic X-ray tomography in low metric regularity assuming that the X-ray transform~$I$ acts on~$L^2(M)$, since then $N = I^\ast I$ is in fact the normal operator for the X-ray transform~$I$.

\begin{proposition}
\label{prop:on-l2}
Let $(M,g)$ be a simple manifold with $g \in C^k(M)$ for some $k \ge 2$. Then $I^\ast I = N$ on~$L^2(M)$.
\end{proposition}

\begin{theorem}
\label{thm:xrt-injective}
Let $(M,g)$ be a simple manifold, $n \coloneqq \dim M \ge 2$ and $g \in C^k(M)$ for some $k \ge 8 + n$. Then the X-ray transform~$I$ is injective on $L^2(M)$.
\end{theorem}

The proofs of the theorems rely on microlocal tools. We study the so-called normal operator $N = I^\ast I$ related to the X-ray transform~$I$. We prove that~$N$ is a non-smooth elliptic operator and construct a principal parametrix with an error term smoothing of order $\tau \in (0,1)$. The construction and its implications use a non-smooth microlocal calculus developed in~\cite{marschall1996} and, in particular, we use the non-smooth symbol and operator classes, continuous Sobolev mapping properties and a commutator theorem there introduced. The details are recalled in section~\ref{sec:preliminaries}.

\subsection{Related results}

The geodesic X-ray transform on a Riemannian manifold has been studied in a variety of contexts and with a variety of tools~\cite{PSUbook,sharafutdinov1994,IlmaMon2019,PSU2014}. The current article focus on the aspect of not studying the X-ray transform directly but via the related normal operator. This approach has seen plenty of applications in $C^\infty$-smooth metric regularity.

In~\cite{PesUhl2005} it was proved that the normal operator on a simple Riemannian manifold is an elliptic pseudodifferential operator in the interior of the manifold --- a result that is essential in their proof that all two dimensional simple Riemannian manifolds are boundary rigid. The normal operator has also played a role in later developments in boundary rigidity~\cite{StefUhl2004,StefUhl2005}. Microlocal methods in relation to the normal operator are useful in geometries permitting conjugate points~\cite{StefUhl2008,StefUhl2012,HolUhl2018}. More recently, there has been interest in isomorphic mapping properties of the normal operator and its variants between suitably weighted function spaces~\cite{MMZ2023,MNP2019}.

Microlocal analysis of the normal operator in the X-ray tomography is in non-smooth geometries virtually unexplored. However, injectivity for the X-ray transform of Lipschitz scalar and $C^{1,1}$ tensor fields on simple $C^{1,1}$ manifolds was proved in two recent articles~\cite{IlmKyk2021,IlmKyk2023}, and injectivity is known for the scalar transform on spherically symmetric $C^{1,1}$ manifolds satisfying the Herglotz condition~\cite{DHI2017}.

The current article uses non-smooth microlocal methods. As references on pseudodifferential operators with symbols non-smooth in both variables we mention~\cite{IvecVoj2022,marschall1996} and as references to paradifferential methods we mention~\cite{Taylor1997}.

\subsection{Acknowledgements}

JI was supported by the Research Council of Finland (grant 351665).
AK was supported by the Research Council of Finland (grant 351656) and by the Finnish Academy of Science and Letters.
KL was supported by NSF.
We thank John M. Lee, Gabriel P. Paternain, Mikko Salo, Hart F. Smith, and Gunther Uhlmann for discussions.

\section{Preliminaries}
\label{sec:preliminaries}

In this section we introduce the geometric set-up, the function spaces, and the operators used throughout the article. 
We also recall the parts of the non-smooth calculus and theorems from~\cite{marschall1996} that are required for the proofs of our main results.

\subsection{Simple manifolds}

In this section we recall the geometric set-up in which we study geodesic X-ray transforms.
Since the Riemannian metrics we consider are not $C^\infty$-smooth, we include the following definition for clarity.

\begin{definition}
\label{def:simple}
Let $k$ be an integer so that $k \ge 2$. Let~$M$ be a compact smooth manifold with a smooth boundary and equip~$M$ with a~$C^k$ smooth Riemannian metric~$g$. We say that $(M,g)$ is simple if~$M$ is $C^k$-diffeomorphic to the closed Euclidean unit ball in~$\R^n$ and the following hold:
\begin{enumerate}
    \item The boundary~$\partial M$ is strictly convex in the sense of the second fundamental form.
    \item The manifold is non-trapping, i.e., all geodesics hit the boundary in a finite time.
    \item There are no conjugate points in~$M$.
\end{enumerate}
\end{definition}

When the Riemannian metric~$g$ is $C^\infty$-smooth, definition~\ref{def:simple} is equivalent to any standard definition of a simple manifold.

\begin{remark}
\label{rem:to-the-ball}
Our analysis of the non-smooth operators is carried out on the closed Euclidean unit ball, which allows us to use smooth global coordinates on our manifold. This allows us to use smooth functions on the manifold without having to worry about limitations on regularity indices. However, we have to interpret our results in the original manifold via a $C^k$-diffeomorphism which restrict the meaningful range of any regularity indices (H\"older or Sobolev) to $[-k,k]$ in the up coming sections.
\end{remark}

To prove Theorem~\ref{thm:xrt-injective} we will use~\cite[Theorem 1]{IlmKyk2021}. There the authors use a slightly different notion is simplicity, but definition~\ref{def:simple} is equivalent to their definition for Riemannian metrics $g \in C^k(M)$ when $k \ge 10$ which holds in the case of Theorem~\ref{thm:xrt-injective}. The proof of equivalence of definitions in~\cite[Theorem 2]{IlmKyk2021} carries over to our simple Riemannian metrics $g \in C^k(M)$ for $k \ge 10$ by the arguments given in~\cite{IlmKyk2021} and since we assume that~$M$ is $C^k$-diffeomorphic to the closed unit ball in~$\R^n$.

Since the conditions defining a simple manifold $(M,g)$ with $g \in C^k(M)$ are open, there is a small open extension $M \subseteq U \subseteq \R^n$ and an extension~$\tilde g$ of~$g$ so that $(\overline{U},\tilde g)$ is a simple manifold with $\tilde g \in C^k(\overline{U})$. For details on the existence of simple extension we refer the reader to~\cite[Proposition 3.8.7]{PSUbook}.

\subsection{Function spaces}

Our definition of a simple manifold includes global coordinates.
Therefore no partitions of unity are needed and the definitions of some operators and function spaces are somewhat simplified.

Let $(M,g)$ be a simple manifold where $g \in C^k(M)$ for some $k \ge 2$. Since~$M$ is $C^k$-diffeomorphic to the closed Euclidean unit ball $\overline{B} \subseteq \R^n$ we take $M = \overline{B}$ from now on and all computations are to be interpreted via a $C^k$-diffeomorphism as explained in remark~\ref{rem:to-the-ball}.

We use smooth global coordinates $(x^1,\dots,x^n)$ in the definitions of our functions spaces. We use the Riemannian volume for~$\der\Vol_g$ to define~$L^2(M)$ in the standard way i.e. $L^2(M) = L^2(M,\der\Vol_g)$.

For $s>0$ we denote by $H^s_c(M)$ the space of compactly supported functions in $H^s(M)$. For $s > 0$ we let $H^{-s}(M)$ be the continuous dual of $H^s(M)$ and $H^{-s}_c(M)$ be the subspace of compactly supported distributions.

Similarly, we define the Zygmund space $C^r_\ast(M)$ to be the space of continuous functions~$f$ on~$M$ whose zero extension to~$\R^n$ is in~$C^r_\ast(\R^n)$ and the norm of a such functions is its $C^r_\ast(\R^n)$-norm.

\subsection{Geodesic X-ray transforms}

Let $(M,g)$ be a simple manifold where $g \in C^k(M)$ for some $k \ge 2$. For a given unit vector $v \in T_xM$ there is a unique geodesic~$\gamma_{x,v}$ corresponding to the initial conditions $\gamma_{x,v}(0) = x$ and $\dot\gamma_{x,v}(0) = v$. Since the manifold is non-trapping, the geodesic~$\gamma_{x,v}$ is defined on a maximal interval of existence $[-\tau_-(x,v),\tau_+(x,v)]$ where $\tau_\pm(x,v) \ge 0$ and we abbreviate $\tau \coloneqq \tau_+$.

The X-ray transform~$If$ of a function $f \in L^2(M)$ is defined for all inwards pointing unit vectors $(x,v) \in \doo{in}SM$ by the formula
\begin{equation}
\label{eqn:back-projection}
If(x,v)
\coloneqq
\int_0^{\tau(x,v)}
f(\gamma_{x,v}(t))
\,\der t.
\end{equation}
The backprojection~$I^\ast h$ of a function~$h$ on~$L^2(\doo{in}(SM))$ is defined for all $x \in M$ by the formula
\begin{equation}
\label{eqn:normal-operator}
I^\ast h(x)
\coloneqq
\int_{S_xM}
h(\phi_{-\tau(x,-v)}(x,v))
\,\der S_x(v).
\end{equation}

Finally, we define the operator~$N$ which we will call the normal operator and which will be the main focus of our study. The normal operator is defined on~$L^2(M)$ by the formula
\begin{equation}
Nf(x)
=
2\int_{S_xM}
\int_0^{\tau(x,v)}
f(\gamma_{x,v}(t)))
\,\der t
\,\der S_x(v).
\end{equation}
We will prove in proposition~\ref{prop:on-l2} that~$N$ agrees with the composition~$I^\ast I$ on~$L^2(M)$, justifying calling it the normal operator.

\subsection{Non-smooth operators and symbols}
\label{sec:non-smooth-oper}

In this section we recall the basics of a non-smooth pseudodifferential calculus introduced in~\cite{marschall1996}. We rerecord the results that are relevant to the current work for the convenience of the reader.

Let $m \in \R$ and $r,L \in \N$ be given.
Multi-indices in~$\N^n$ are denoted by~$\alpha$ and~$\beta$. For all $\rho,\delta \in [0,1]$ the symbol class~$S^m_{\rho\delta}(r,L)$ consists of continuous functions $p \colon \R^n \times \R^n \to \R$ satisfying the estimates
\begin{equation}
\abs{\partial_\xi^\alpha p(x,\xi)}
\le
C_\alpha
(1 + \abs{\xi})^{m - \rho\abs{\alpha}}
\end{equation}
and
\begin{equation}
\norm{\partial_\xi^\alpha p(\,\cdot\,,\xi)}_{C^r_\ast}
\le
C_{\alpha r}
(1 + \abs{\xi})^{m + r\delta - \rho\abs{\alpha}}
\end{equation}
for all $\abs{\alpha} \le L$.

Given a symbol $p \in S^m_{\rho\delta}(r,L)$ the corresponding operator~$\Op(p)$ is defined by its action
\begin{equation}
\Op(p)f(x)
=
\int_{\R^n}
e^{ix\cdot\xi}
p(x,\xi)
\hat f(\xi)
\,\der \xi
\end{equation}
on functions~$f$ in~$L^2(\R^n)$. The identity operator~$\Id$ is the operator corresponding to the constant symbol~$1$.

We end the preliminaries by isolating two useful results on the operators of class $\Psi^m(r,L)$. For the proofs of the lemmas we refer the reader to~\cite{marschall1996}.

\begin{lemma}[\cite{marschall1996} Theorem 2.1.]
\label{lma:sobolev-property}
Let $p \in S^m_{\rho\delta}(r,L)$ and consider the operator $P \coloneqq \Op(p)$. Suppose that $\rho,\delta \in [0,1]$ and $r,L > 0$ satisfy
\begin{equation}
\delta
\le
\rho,
\quad
L
>
\frac{n}{2},
\quad
r
>
\frac{1 - \rho}{1 - \delta}
\frac{n}{2}.
\end{equation}
Then the operator $P \colon H^{s+m}(\R^n) \to H^s(\R^n)$ is bounded when
\begin{equation}
(1-\rho)
\frac{n}{2}
-
(1-\delta)
r
<
s
<
r.
\end{equation}
\end{lemma}

\begin{lemma}[\cite{marschall1996} Theorem 3.5.]
\label{lma:composition}
Let $p \in S^{m_1}_{\rho_1\delta_1}(r,L)$ and $q \in S^{m_2}_{\delta_2\rho_2}(r,L+\frac{n}{2}+1)$ and suppose that $\delta_1 < \rho_2$ and $L > \frac{n}{2}$. Denote the corresponding operators by $P \coloneqq \Op(p)$ and $Q \coloneqq \Op(q)$. Let $\tau \in (0,1]$ be such that $0 < \tau < r$. Define
\begin{equation}
\delta
\coloneqq
\max\{\delta_1+(\rho_1-\delta_2)\tau,\delta_2\}
\quad
\text{and}
\quad
\rho
\coloneqq
\min\{\rho_1,\rho_2\}.
\end{equation}
Assume that $\delta \le \rho$ and in the case $\rho < 1$ suppose in addition that $r > \frac{1-\rho}{1-\delta}\frac{n}{2} + \tau$. Then the commutator
\begin{equation}
QP
-
\Op(qp)
\colon
H^{s+m_1+m_2-(\rho_1-\delta_2)\tau}(\R^n)
\to
H^s(\R^n)
\end{equation}
is bounded when
\begin{equation}
\max\{-m_2,0 \}
+
(1-\rho)\frac{n}{2}
-
(1-\delta)(r-\tau)
<
s
<
r-\max \{m_2, 0\}.
\end{equation}
\end{lemma}

\section{Parametrix construction for the normal operator}

This section provides a detailed analysis of the operator~$N$ culminating in a leading order parametrix construction in the non-smooth symbol calculus presented in section~\ref{sec:non-smooth-oper}. The parametrix construction is the main tool used in the proofs of our main theorems.

\subsection{The Schwartz kernel and the symbol}

The objective of this section is to study the operator~$N$ as a non-smooth elliptic pseudodifferential operator. We begin from the Schwartz kernel of the operator and analyse its symbol by dissecting it into manageable parts. The end result containing the principal part of the symbol is presented in corollary~\ref{cor:classes-of-N}.

\begin{lemma}
Let $(M,g)$ be a simple manifold with $g \in C^k(M)$ for some $k \ge 2$. Let $a(x,y) = \det(\der\exp_x|_{\exp_x^{-1}(y)})^{-1}$. Then for all $f \in L^2(M)$ we have
\begin{equation}
Nf(x)
=
2\int_M
a(x,y)
d_g(x,y)^{1-n}
f(y)
\,\der\Vol_g(y).
\end{equation}
\end{lemma}

\begin{proof}
The same formula is derived in~\cite[Lemma 8.1.10]{PSUbook} when $g \in C^\infty(M)$. The computation works when $g \in C^k(M)$ with $k \ge 2$. 
\end{proof}

The Schwartz kernel of the operator $N$ is
\begin{equation}
K(x,y)
=
2a(x,y)
d_g(x,y)^{1-n}
\end{equation}
on $M \times M$. We will construct leading order parametrices for operators on $\R^n$ related to the Schwartz kernels of the form
\begin{equation}
\label{eqn:tilde-K}
\tilde K(x,y)
\coloneqq
\psi(x)
2a(x,y)
d_g(x,y)^{1-n}
\det(g(y))^{\frac12}
\phi(y)
\end{equation}
where $\psi$ and $\phi$ are suitable cut-off functions in $\R^n$.

Consider $\Omega \subseteq M$ and consider $f \in H^s_c(M)$ so that $\supp f \subseteq \Omega$. We can choose a cut-off function $\phi \in C^\infty_c(M)$ so that $\phi f = f$ on $M$. Then if $\psi \in C^\infty_c(M)$ is to that $\psi = 1$ on $\Omega$ we have for all $x \in \Omega$ that
\begin{equation}
\begin{split}
Nf(x)
&=
\int_{\R^n}
\psi(x)K(x,y)\det(g(y))^{\frac12}\phi(y)f(y)
\,\der y
\\
&=
\int_{\R^n}
\tilde K(x,y)f(y)
\,\der y.
\end{split}
\end{equation}
We let $\tilde N$ be the operator corresponding to the kernel $\tilde K$. Then $Nf(x) = \tilde Nf(x)$ on $\Omega$ which shows that it is enough to only consider operators with kernel of the form~\eqref{eqn:tilde-K}. For the details see the proof of Theorem~\ref{thm:elliptic-regularity} in section~\ref{sec:proofs-of-main-theorems}. From now on we let $N = \tilde N$ to avoid cluttered notation and we keep the cut-off functions $\psi$ and $\phi$ fixed for the remainder of this section.

We will prove that $N \in \Psi^{-1}(k-s,s-4)$ for all $s \in \N$ with $4 \le s \le k$. This is accomplished by studying the operator in the global coordinates of the Euclidean unit ball and by computing the symbol of the operator.
By \cite[Lemma 8.1.12]{PSUbook} we can write in the coordinates that
\begin{equation}
\label{eqn:kernel-of-N}
\tilde K(x,y)
=
\psi(x)
\frac{2a(x,y)\det(g(y))^{1/2}}{[G_{jk}(x,y)(x-y)^j(x-y)^k]^{\frac{n-1}{2}}}
\phi(y)
\end{equation}
for some functions~$G_{jk}$ with $G_{jk}(x,x) = g_{jk}(x)$.

\begin{lemma}
\label{lma:kernel-regularity}
Let $(M,g)$ be a simple manifold with $g \in C^k(M)$ for some $k \ge 3$. Then $\tilde K \in C^{k-2}(\R^n \times \R^n \setminus \Delta)$ where $\Delta \coloneqq \{(x,x)\,:\,x \in \R^n\}$ is the diagonal in $\R^n \times \R^n$.
\end{lemma}

\begin{proof}
The kernel $\tilde K$ can be expressed in the form
\begin{equation}
\label{eqn:kernel-regularity}
\tilde K(x,y)
=
\psi(x)
2a(x,y)
d_g(x,y)^{1-n}
\det(g(y))^{\frac12}
\phi(y).
\end{equation}
By standard ODE theory the geodesic flow has $C^{k-1}$ smooth initial value dependence when $g \in C^k(M)$, and thus the exponential function is also $C^k$. It follows that $a \in C^{k-2}(M \times M)$. In addition, since $d_g(x,\exp_x(v)) = \abs{v}_g$ for $(x,v) \in TM$ it follows that $d_g(x,y) \in C^{k-1}(M \times M \setminus \Delta)$. Finally, since the determinant term in~\eqref{eqn:kernel-regularity} is $C^k$ we see that $\tilde K$ is $C^{k-2}$ off diagonal as claimed.
\end{proof}

By denoting
\begin{equation}
\label{eqn:z-kernel}
k(x,z) \coloneqq \tilde K(x,x-z)
\end{equation}
and letting
\begin{equation}
\label{eqn:full-symbol}
a(x,\xi)
\coloneqq
\int_{\R^n}
e^{-iz \cdot \xi}
k(x,z)
\,\der z
\end{equation}
the normal operator on~$L^2(M)$ can be brought to the form
\begin{equation}
Nf(x)
=
\int_{R^n}
e^{ix \cdot \xi}
a(x,\xi)
\hat f(\xi)
\,\der \xi.
\end{equation}

The following lemma is a finite regularity adaptation of the classical result~\cite[Chapter VI.7.4]{Stein1993}.

\begin{lemma}
\label{lma:stein}
Let $m < 0$ and suppose that $\kappa \in C^l_c(\R^n \times (\R^n\setminus\{0\}))$ where $l \in \N$ satisfies estimates
\begin{equation}
\abs{\partial^\alpha_x\partial^\beta_z\kappa(x,z)}
\le
C_{\alpha\beta}
\abs{z}^{-m-n-\abs{\beta}},
\quad
z \ne 0,
\end{equation}
for $\abs{\alpha} + \abs{\beta} \le l$. Then the function on $\R^n \times \R^n$ defined by
\begin{equation}
b(x,\xi)
\coloneqq
\int_{\R^n}
e^{-iz \cdot \xi}
\kappa(x,z)
\,\der z
\end{equation}
is a symbol in the class $S^m(l-s,s-2)$ for all $s \in \N$ with $2 \le s \le l$.
\end{lemma}

\begin{proof}
Since by assumption
\begin{equation}
\abs{\partial_z^\beta\kappa(x,z)} \le C_\beta\abs{z}^{-m-n-\abs{\beta}}, \quad z \ne 0,
\end{equation}
holds for all $\abs{\beta} \le l$ and since~$\kappa$ is compactly supported, it can be shown by using~\cite[VI 4.5.]{Stein1993} as in~\cite[VI 7.4.]{Stein1993} that~$b$ is a continuous function on $\R^n \times \R^n$ and
\begin{equation}
\label{eqn:first-symbol-estimate}
\abs{\partial_\xi^\beta b(x,\xi)}
\le
C_\beta(1+\abs{\xi})^{m-\abs{\beta}}
\end{equation}
for all $\abs{\beta} \le l-2$, which is the first estimate we set out to prove.

Then let $s \in [2,l]$ be an integer. Since~$\kappa$ is compactly supported we have
\begin{equation}
\partial_x^\alpha b(x,\xi)
=
\int_{\R^n}
e^{-iz \cdot \xi}
\partial_x^\alpha
\kappa(x,z)
\,\der z.
\end{equation}
Let us denote $\kappa_\alpha(x,z) = \partial_x^\alpha\kappa(x,z)$. Then it holds that
\begin{equation}
\abs{\partial_z^\beta\kappa_\alpha(x,z)}
\le
C_{\alpha\beta}
\abs{z}^{-m-n-\abs{\beta}},
\quad
z\ne 0,
\end{equation}
for all $\abs{\beta} \le l - \abs{\alpha}$. Therefore by a similar application of~\cite[VI 4.5]{Stein1993} we have
\begin{equation}
\abs{\partial_\xi^\beta\partial_x^\alpha b(x,\xi)}
\le
C_{\alpha\beta}(1+\abs{\xi})^{m-\abs{\beta}}
\end{equation}
for all $\abs{\alpha} + \abs{\beta} \le l - 2$. Then it follows that
\begin{equation}
\label{eqn:second-symbol-estimate}
\norm{\partial_\xi^\beta b(\,\cdot\,,\xi)}_{C^{l-s}_\ast}
\le
\norm{\partial_\xi^\beta b(\,\cdot\,,\xi)}_{C^{l-s}}
\le
C_{\alpha\beta}(1+\abs{\xi})^{m-\abs{\beta}},
\end{equation}
which uses compactness of the support of~$\kappa$ again. By estimates~\eqref{eqn:first-symbol-estimate} and~\eqref{eqn:second-symbol-estimate} we have shown $b \in S^{m}(l-s,s-2)$ for all integers $s \in [2,l]$ as claimed.
\end{proof}

\begin{lemma}
\label{lma:full-symbol}
Let $(M,g)$ be a simple manifold with $g \in C^k(M)$ for some $k \ge 5$. Then the function~$a$ defined by~\eqref{eqn:full-symbol} belongs to $S^{-1}(k-s,s-4)$ for all $s \in [4,k]$ with $4 \le s \le k$.
\end{lemma}

\begin{proof}
We write the kernel in~\eqref{eqn:z-kernel} in the form
\begin{equation}
\begin{split}
k(x,z)
=
\abs{z}^{1-n}
\psi(x)
\frac{2a(x,x-z)\det(g(x-z))^{\frac12}}{[G_{jk}(x,x-z)\frac{z^j}{\abs{z}}\frac{z^k}{\abs{z}}]^{\frac{n-1}{2}}}
\phi(x-z)
\end{split}
\end{equation}
and denote
\begin{equation}
k_0(x,z)
=
\psi(x)
\frac{2a(x,x-z)\det(g(x-z))^{\frac12}}{[G_{jk}(x,x-z)\frac{z^j}{\abs{z}}\frac{z^k}{\abs{z}}]^{\frac{n-1}{2}}}
\phi(x-z).
\end{equation}
Then $k_0(x,z)$ is $C^{k-2}$ for $z \ne 0$ and its derivatives $\partial_x^\alpha\partial_z^\beta k_0(x,z)$, $z \ne 0$, are bounded for all $\abs{\alpha} + \abs{\beta} \le k-2$ since $k_0(x,z)$ is compactly supported. Thus~$k$ satisfies estimates
\begin{equation}
\abs{\partial_{x}^\alpha\partial_z^\beta k(x,z)}
\le
C_{\alpha\beta}
\abs{z}^{1-n-\abs{\beta}},
\quad
z \ne 0,
\end{equation}
for all $\abs{\alpha} + \abs{\beta} \le k-4$ and the claim follows from lemma~\ref{lma:stein}.
\end{proof}

\begin{remark}
The symbol $a(x,\xi)$ is smooth in~$\xi$ but our argument does not prove that~$a(x,\xi)$ satisfies the estimates of the class~$S^m(k,L)$ for all orders~$L$ of $\xi$-derivatives. Thus we cannot use paradifferential calculus to study~$N$.
\end{remark}

Lemma~\ref{lma:full-symbol} shows that $N \in \Psi^{-1}(k-s,s-4)$ for all $s \in \N$ with $4 \le s \le k$ when the Riemannian metric is in~$C^k(M)$ when $k \ge 5$. The rest of this section is devoted to computing the principal symbol of the normal operator. We start by writing the kernel~$k$ as 
\begin{equation}
k(x,z)
=
\abs{z}^{1-n}
h
\left(
x,z,\frac{z}{\abs{z}}
\right)
\end{equation}
where~$h$ is a function on $\R^n \times [0,\infty) \times S^{n-1}$ defined by
\begin{equation}
h(x,r,\omega)
=
\psi(x)
\frac{2a(x,x-r\omega)\det(g(x-r\omega))^{\frac12}}{[G_{jk}(x,x-r\omega)\omega^j\omega^k]^{\frac{n-1}{2}}}
\phi(x-r\omega).
\end{equation}
Since $G_{jk}(x,x-r\omega)\omega^j\omega^k$ is non-vanishing we see that $h \in C^{k-2}(\R^n \times [0,\infty) \times S^{n-1})$. By the Fundamental theorem of calculus
\begin{equation}
h(x,r,\omega)
=
h(x,0,\omega)
+
r
\int_0^1
\partial_rh(x,rt,\omega)
\,\der t
\end{equation}
and we can decompose $k(x,z) = k_{-1}(x,z) + r(x,z)$ where
\begin{equation}
k_{-1}(x,z)
\coloneqq
\abs{z}^{1-n}
h
\left(
x,0,\frac{z}{\abs{z}}
\right)
\end{equation}
and
\begin{equation}
r(x,z)
\coloneqq
\abs{z}^{2-n}
\int_0^1
\partial_r
h
\left(
x,\abs{z}t,\frac{z}{\abs{z}}
\right)
\,\der t.
\end{equation}
Since $k(x,z)$ is compactly supported in~$z$ we can choose a cut-off function~$\chi(z)$ so that $0 \le \chi \le 1$ and $\chi = 1$ near the origin so that
\begin{equation}
k(x,z)
=
\chi(z)
k(x,z)
=
\chi(z)
k_{-1}(x,z)
+
\chi(z)
r(x,z).
\end{equation}
Now the full symbol of~$N$ is decomposed as
\begin{equation}
\begin{split}
a(x,\xi)
&=
\mathcal{F}(\chi(\,\cdot\,)k_{-1}(x,\,\cdot\,))(\xi)
+
\mathcal{F}(\chi(\,\cdot\,)r(x,\,\cdot\,))(\xi)
\\
&\eqqcolon
a_{-1}(x,\xi)
+
c(x,\xi).
\end{split}
\end{equation}

\begin{lemma}
\label{lma:principal-part}
Let $(M,g)$ be a simple manifold with $g \in C^k(M)$ for some $k \ge 5$. Then $a_{-1} \in S^{-1}(k-s,s-4)$ for all $s \in \N$ with $4 \le s \le k$ and~$c \in S^{-2}(k-s,s-5)$ for all $s \in \N$ with $5 \le s \le k$.
\end{lemma}

\begin{proof}
Since $h \in C^{k-2}(\R^n \times [0,\infty) \times S^{n-1})$ is compactly supported in~$x$ and~$r$ and~$S^{n-1}$ is compact, we can extend~$h$ to a compactly supported function on $\R^n \times [0,\infty) \times S^{n-1}$. Thus $\partial^\alpha_x\partial^l_r\partial^\beta_\omega h(x,r,\omega)$ is continuous and compactly supported for all $\alpha \in \N^n$, $\beta \in \N^{n-1}$ and $l \in \N$ for which we have $\abs{\alpha} + l + \abs{\beta} \le k-2$.

First, we prove the claim about the Fourier transform~$c$ of the remainder. 
Since derivatives of~$h$ are continuous and compactly supported, a simple computation using the chain rule shows that
\begin{equation}
\begin{split}
\abs{
\partial_x^\alpha
\partial_{z^j}
\partial_rh
\left(
x,\abs{z}t,\frac{z}{\abs{z}}
\right)
}
&\le
C\abs{z}^{-1}
\end{split}
\end{equation}
near $z = 0$ and for all $t \in [0,1]$ when $\abs{\alpha} + 2 \le k-2$. Therefore by iteration
\begin{equation}
\abs{
\partial_x^\alpha\partial_z^\beta
\partial_rh
\left(
x,\abs{z}t,\frac{z}{\abs{z}}
\right)
}
\le
C_{\alpha\beta}
\abs{z}^{-\abs{\beta}}
\end{equation}
near $z = 0$ when $\abs{\alpha} + \abs{\beta} + 1 \le k-2$. The above estimate applied to the remainder term $r(x,z)$ yields
\begin{equation}
\abs{
\partial_x^\alpha
\partial_z^\beta
\int_0^1
\partial_r
h
\left(
x,\abs{z}t,\frac{z}{\abs{z}}
\right)
\,\der t
}
\le
C_{\alpha\beta}
\abs{z}^{-\abs{\beta}}
\end{equation}
near $z = 0$ when $\abs{\alpha} + \abs{\beta} + 1 \le k-2$, which implies that
\begin{equation}
\abs{
\partial_x^\alpha
\partial_z^\beta
(\chi(z)r(x,z))
}
\le
C_{\alpha\beta}
\abs{z}^{2-n-\abs{\beta}}
\end{equation}
for all~$z$ and $\abs{\alpha} + \abs{\beta} \le k-3$ since the cut-off~$\chi(z)$ implies that only have to derive the estimate near $z = 0$. It follows from lemma~\ref{lma:stein} that $c \in S^{-2}(k-s,s-5)$ for all $s \in \N$ with $5 \le s \le k$.

By a similar computation we see that~$k_{-1}$ satisfies estimates
\begin{equation}
\abs{
\partial_x^\alpha
\partial_z^\beta
(\chi(z)k_{-1}(x,z))
}
\le
C_{\alpha\beta}
\abs{z}^{1-n-\abs{\beta}}
\end{equation}
for all~$z$ and $\abs{\alpha} + \abs{\beta} \le k-2$.
Thus, again, by lemma~\ref{lma:stein} we have $a_{-1} \in S^{-1}(k-s,s-4)$ for all $s \in \N$ with $4 \le s \le k$, which finishes the proof.
\end{proof}

In the next section we construct a leading order parametrix for~$N$. To this end we need to find a more explicit representation for~$a_{-1}$. We write $\chi(z)k_{-1}(x,z) = k_{-1}(x,z) - (1-\chi(z))k_{-1}(x,z)$ and analyze the Fourier transforms of the parts separately.

\begin{lemma}
\label{lma:singularity}
For a dimensional constant $C$ it holds that
\begin{equation}
\int_{\R^n}
e^{-iz \cdot \xi}
k_{-1}(x,z)
\,\der z
=
C
\psi(x)
\abs{\xi}^{-1}_{g(x)}
\phi(x).
\end{equation}
\end{lemma}

\begin{proof}
The Fourier transform of
\begin{equation}
k_{-1}(x,z)
=
\psi(x)
\frac{2\det(g(x))^{1/2}}{(g_{jk}(x)z^jz^k)^{\frac{n-1}{2}}} 
\phi(x)
\end{equation}
in~$z$ is computed in~\cite[Chapter 8.1]{PSUbook}. The only difference is regularity in~$x$, which does not affect the computation.
\end{proof}

\begin{lemma}
\label{lma:second-part-of-the-psymbol}
Let $(M,g)$ be a simple manifold with $g \in C^k(M)$ for some $k \ge 3$. Let $b(x,\xi) \coloneqq \mathcal{F}((1-\chi(\,\cdot\,))k_{-1}(x,\,\cdot\,))(\xi)$. Then $b \in C^{k-2}_xC^\infty_\xi(\R^n \times (\R^n \setminus \{0\}))$. Moreover,~$b$ has a singularity of type $\abs{\xi}_g^{-1}$ at the origin, and  satisfies $\abs{b(x,\xi)} \le C\abs{\xi}^{2-k}$ when~$\abs{\xi}$ is large enough.
\end{lemma}

\begin{proof}
The fact that $b(x,\xi)$ has a singularity of type $\abs{\xi}_{g(x)}^{-1}$ at the origin follows from the fact that $b(x,\xi) = a(x,\xi) - C\psi(x)\abs{\xi}^{-1}_{g(x)}\phi(x)$ near $\xi = 0$ and $a \in C^{k-2}_xC^\infty_\xi(\R^n \times \R^n)$.

Next, we prove the claim about the decay of~$b$ away from $\xi = 0$. Since $(1-\chi(z))k_{-1}(x,z) = 0$ for~$z$ near the origin and since
\begin{equation}
(1-\chi(z))
k_{-1}(x,z)
=
(1-\chi(z))
\psi(x)
\frac{2\det(g(x))^{1/2}}{(g_{jk}(x)z^jz^k)^{\frac{n-1}{2}}}
\phi(x)
\end{equation}
for $z \ne 0$, we know that $(1-\chi)k_{-1}$ is in $C^{k-2}(\R^n \times (\R^n\setminus \{0\}))$ and compactly supported in~$x$. As in the proof of lemma~\ref{lma:principal-part} we can use boundedness of the derivatives of $h(x,0,z\abs{z}^{-1})$ to prove that
\begin{equation}
\abs{
\partial_z^\alpha k_{-1}(x,z)
}
=
\abs{
\partial_z^\alpha 
\left(
(1-\chi(z))
\abs{z}^{1-n}
h
\left(
x,0,\frac{z}{\abs{z}}
\right)
\right)
}
\le
C_{\alpha}
\abs{z}^{-n-1}
\end{equation}
for $2 \le \abs{\alpha} \le k-2$ which proves that for a fixed~$x$ we have $\partial_z^\alpha k_{-1}(x,z) \in L^1(\R^n)$. Therefore by the Riemann--Lebesgue lemma we conclude that
\begin{equation}
\label{eqn:riem-leb-application}
\begin{split}
&
\abs{\xi^\alpha
\mathcal{F}((1-\chi(\,\cdot\,))k_{-1}(x,\,\cdot\,))(\xi)
}
\\&\quad
=
\abs{
\mathcal{F}(\partial^\alpha_z((1-\chi(\,\cdot\,))k_{-1}(x,\,\cdot\,)))(\xi)
}
\\&\quad
\to
0
\end{split}
\end{equation}
for all $2 \le \abs{\alpha} \le k-2$ as $\abs{\xi} \to \infty$. Thus since~\ref{eqn:riem-leb-application} holds for all $2 \le \abs{\alpha} \le k-2$ we have $\abs{b(x,\xi)} \le C\abs{\xi}^{2-k}$ for~$\abs{\xi}$ large enough.
\end{proof}

Lemmas~\ref{lma:principal-part},~\ref{lma:singularity} and~\ref{lma:second-part-of-the-psymbol} together prove the following corollary.

\begin{corollary}
\label{cor:classes-of-N}
Let $(M,g)$ be a simple manifold with $g \in C^k(M)$ for some $k \ge 5$. Then~$N \in \Psi^{-1}(k-s,s-4)$ for all $s \in \N$ with $4 \le s \le k$. The principal symbol of~$N$ is
\begin{equation}
\label{eqn:principal-symbol}
a_{-1}(x,\xi) = C\psi(x) \abs{\xi}^{-1}_g \phi(x) - b(x,\xi) \in S^{-1}(k-s,s-4)
\end{equation}
where~$b$ is as in lemma~\ref{lma:second-part-of-the-psymbol} and $s \in \N$ with $4 \le s \le k$, in particular this shows that~$N$ is elliptic of order $-1$ in the sense of principal symbol.
\end{corollary}

The function~$a_{-1}$ is a function on the whole cotangent bundle and thus~$b$ has to have a singularity of type $\abs{\xi}^{-1}_{g(x)}$ at $\xi = 0$ to cancel out the singularity in $C\psi(x)\abs{\xi}^{-1}_g\phi(x)$.

\subsection{Parametrix construction}

In this section we construct a leading order parametrix for the normal operator. The construction is based on a commutator result in~\cite{marschall1996}. We define $p(x,\xi) \coloneqq C^{-1}\zeta(\xi)\abs{\xi}_{g(x)}$ for some $\zeta \in C^\infty(\R^n)$ so that $0 \le \zeta \le 1$, $\zeta = 0$ near $\xi = 0$ and $\zeta = 1$ for large~$\xi$, and where~$C$ is the same dimensional constant as in lemma~\ref{lma:singularity}. We will prove that the operator corresponding to the symbol~$p$ which is in $S^1(k-s,N)$ for all $s \in \N$ with $4 \le s \le k$ and $N \in \N$ provides the parametrix to the leading order.

\begin{lemma}
\label{lma:commutator-application}
Let $(M,g)$ be a simple manifold with $g \in C^k(M)$ for some $k \ge 7 + \frac{n}{2}$. Let $P = \Op(p)$. If $\tau \in (0,1]$ is fixed then the operator
\begin{equation}
PN
-
\Op(pa)
\colon
H^{t-\tau}(\R^n)
\to
H^t(\R^n)
\end{equation}
is continuous when $-(1-\tau)(k-5-\frac{n}{2}-\tau) < t < k-6-\frac{n}{2}$.
\end{lemma}

\begin{proof}
Choose $s \in \N$ so that $s \in (4+\frac{n}{2},k-1)$ which is possible since $k \ge 7 + \frac{n}{2}$. Let $L \coloneqq s-4$ and let $r \coloneqq k-s$. Then $L > \frac{n}{2}$ and $r > 1 \ge \tau$. By lemma~\ref{lma:full-symbol} we have $N \in \Psi^{-1}(r,L)$ and also it holds that $P \in \Psi^1(r,L+1+\frac{n}{2})$, which means that we are in the setting of lemma~\ref{lma:composition}. For~$\delta$ and~$\rho$ as the lemma it holds that
\begin{equation}
\delta
=
\tau,
\quad
\rho
=
1
\quad
\text{and}
\quad
\delta
<
\rho.
\end{equation}
Thus since $m_1 = -1$ and $m_2 = 1$ in the lemma the commutator
\begin{equation}
PN - \Op(pa)
\colon
H^{t-\tau}(\R^n)
\to
H^t(\R^n)
\end{equation}
is continuous for
\begin{equation}
\max\{-1,0\}
-
(1-\tau)
(r-\tau)
<
t
<
r
-
\max\{1,0\}
\end{equation}
which simplifies to
\begin{equation}
-(1-\tau)
(k-s-\tau)
<
t
<
k-s-1.
\end{equation}
To have a non-empty range of indices~$t$ we must have
\begin{equation}
k-s
>
\frac{1+(1-\tau)\tau}{2-\tau}
\end{equation}
which is satisfied since $s < k-1$ and by an elementary computation it holds that $\frac{1+(1-\tau)\tau}{2-\tau} \le 1$ for all $\tau \in (0,1]$.

Finally to conclude the proof we note that if
\begin{equation}
-(1-\tau)(k-5-\frac{n}{2}-\tau)
<
t
<
k-6-\frac{n}{2}
\end{equation}
there is $s_t \in \N$ so that $s_t \in [4+\frac{n}{2},k-1)$ and
\begin{equation}
-(1-\tau)(k-s_t-\tau)
<
t
<
k-s_t-1
\end{equation}
since $k \ge 7 + \frac{n}{2}$ and thus the operator $PN - \Op(pa) \colon H^{t-\tau}(\R^n) \to H^t(\R^n)$ is continuous as claimed.
\end{proof}

\begin{lemma}
\label{lma:p-aminus}
Let $(M,g)$ be a simple manifold with $g \in C^k(M)$ for some $k \ge 7 + \frac{n}{2}$. Then $\Op(pa_{-1}) = \Id + R_1$ where $\Id$ is an operator acting as the identity on elements in $H^{t+2-k}(\R^n)$ which are supported in the set where $\psi = 1 = \phi$ and the remainder
\begin{equation}
R_1
\colon
H^{t+2-k}(\R^n)
\to
H^t(\R^n)
\end{equation}
is continuous when $-k+2 < t < k-2$.
\end{lemma}

\begin{proof}
By corollary~\ref{cor:classes-of-N} the principal symbol~$a_{-1}$ of~$N$ can be decomposed as
\begin{equation}
a_{-1}(x,\xi)
=
C\psi(x)\abs{\xi}^{-1}_{g(x)}\phi(x)
-
b(x,\xi)
\end{equation}
where $b(x,\xi)$ is in $C^{k-2}_xC^\infty_\xi(\R^n \times (\R^n \setminus \{0\}))$ which is compactly supported in~$x$ and decays faster than~$\abs{\xi}^{2-k}$ in~$\xi$. Therefore
\begin{equation}
\begin{split}
p(x,\xi)a_{-1}(x,\xi)
&=
\zeta(\xi)\psi(x)\phi(x)
-
C^{-1}\zeta(\xi)b(x,\xi)
\\
&=
\psi(x)\phi(x)
-
(1-\zeta(\xi))\psi(x)\phi(x)
-
C^{-1}\zeta(\xi)b(x,\xi).
\end{split}
\end{equation}
Since $(1-\zeta(\xi))\psi(x)\phi(x)$ is smooth and compactly supported, it decays faster than $\abs{\xi}^{-l}$ for any $l \in \N$. Since $\psi(x)\phi(x)$ equals to $1$ on in the set where $\psi = 1 = \phi$ the corresponding operator acts as the identity on functions in $H^{t+2-k}(\R^n)$ which are supported in this set. Also, by lemma~\ref{lma:second-part-of-the-psymbol} the function $\zeta(\xi)b(x,\xi)$ decays faster than~$\abs{\xi}^{2-k}$. Therefore, since the support in~$x$ is compact and $b \in C^{k-2}_xC^\infty_\xi(\R^n\times(\R^n \setminus\{0\}))$ and $\zeta(\xi) = 0$ near $\xi = 0$ it follows from the definitions that
\begin{equation}
\tilde b(x,\xi)
\coloneqq
-
(1-\zeta(\xi))\psi(x)\phi(x)
-
C^{-1}\zeta(\xi)b(x,\xi)
\end{equation}
is a symbol in the class $S^{2-k}(k-2,1+\lfloor\frac{n}{2}\rfloor)$. Therefore by lemma~\ref{lma:sobolev-property} that $\Op(\tilde b) \colon H^{t+2-k}(\R^n) \to H^t(\R^n)$ for all $-k+2 < t < k-2$ since $1 + \lfloor\frac{n}{2}\rfloor > \frac{n}{2}$, which proves the claim.
\end{proof}

\begin{lemma}
\label{lma:p-c}
Let $(M,g)$ be a simple manifold with $g \in C^k(M)$ for some $k \ge 7 + \frac{n}{2}$. Then the operator
\begin{equation}
\Op(pc)
\colon
H^{t-1}(\R^n)
\to
H^t(\R^n)
\end{equation}
is continuous when $-k+6+\frac{n}{2} < t < k-6-\frac{n}{2}$.
\end{lemma}

\begin{proof}
Let $s \in \N$ be so that $s \in (5+\frac{n}{2},k)$. Since~$p$ is in $S^1(k-s,s-5)$ and~$c$ is in $S^{-2}(k-s,s-5)$ by lemma~\ref{lma:principal-part} the product~$pc$ is in $S^{-1}(k-s,s-5)$. Furthermore, since $s-5 > \frac{n}{2}$ it follows from lemma~\ref{lma:sobolev-property} that~$\Op(pc)$ continuously maps from $H^{t-1}(\R^n)$ to $H^t(\R^n)$ for all $-k+s < t < k-s$. To see that the continuous mapping property holds for all $-k+6+\frac{n}{2} < t < k-6-\frac{n}{2}$, we note that given any such~$t$ we can choose any $s_t \in \N$ so that $s_t \in (5+\frac{n}{2},k-t)$ when $t \ge 0$ or $s_t \in (5+\frac{n}{2},k+t)$ when $t < 0$ and it holds that $s_t \in \N$ with $s_t \in (5+\frac{n}{2},k)$ and $-k+s_t < t < k-s_t$. This finishes the proof.
\end{proof}

\begin{lemma}
\label{lma:parametrix}
Let $(M,g)$ be a simple manifold with $g \in C^k(M)$ for some $k \ge 7 + \frac{n}{2}$. Let $P = \Op(p)$. Then there is $\eps > 0$ so that $PN = \Id + R$
where $\Id$ is an operator acting as the identity on elements in $H^{t-\tau}(\R^n)$ which are supported in the set where $\psi = 1 = \phi$ and the remainder
\begin{equation}
R
\colon
H^{t-\tau}(\R^n) 
\to 
H^t(\R^n)
\end{equation}
is continuous whenever $0 < \tau \le \eps$ and
\begin{equation}
-k+6+\frac{n}{2}
<
t
<
k-6-\frac{n}{2}.
\end{equation}
\end{lemma}

\begin{proof}
By lemma~\ref{lma:principal-part} we may write
\begin{equation}
\begin{split}
PN
&=
\Op(pa)
+
(PN - \Op(pa))
\\
&=
\Op(pa_{-1})
+
(PN - \Op(pa))
+
\Op(pc).
\end{split}
\end{equation}
Let $\tau \in (0,1]$. Then by lemma~\ref{lma:commutator-application} we have that
\begin{equation}
\label{eqn:comm}
PN - \Op(pa)
\colon 
H^{t-\tau}(\R^n)
\to
H^t(\R^n)
\end{equation}
is continuous for $-(1-\tau)(k-5-\frac{n}{2}-\tau) < t < k-6-\frac{n}{2}$. By lemmas~\ref{lma:p-aminus} and~\ref{lma:p-c} the operator $\Op(pa_{-1})$ is the identity up to an operator~$R_1$ that is smoothing by~$2$ degrees and~$\Op(pc)$ is smoothing by~$1$ degree, and therefore~$R_1$ and~$\Op(pc)$ are also smoothing by~$\tau$ degrees. More precisely, $\Op(pa_{-1}) = \Id + R_1$, and we have that
\begin{equation}
\label{eqn:r-1}
R_1
\colon
H^{t-\tau}(\R^n)
\to
H^t(\R^n)
\end{equation}
is continuous for $-k+2 < t < k-2$ and
\begin{equation}
\label{eqn:op-pc}
\Op(pc)
\colon
H^{t-\tau}(\R^n)
\to
H^t(\R^n)
\end{equation}
is continuous for $-k+6+\frac{n}{2} < t < k-6-\frac{n}{2}$. Letting~$R$ be the sum of the operators in~\eqref{eqn:comm},~\eqref{eqn:r-1} and~\eqref{eqn:op-pc} we find that $PN = \Id + R$. Now suppose that $\tau$ is close enough to zero. Then the remainder continuously maps $H^{t-\tau}(\R^n)$ to $H^t(\R^n)$ for
\begin{equation}
-k+6+\frac{n}{2}
<
t
<
k-6-\frac{n}{2}
\end{equation}
since $k-6-\frac{n}{2}$ is the smallest among the upper bound requirements and when $\tau$ is close to zero $-k+6+\frac{n}{2}$ is the largest of the lower bound. This proves the claimed identity and the mapping properties.
\end{proof}

\section{Proofs of main theorems}
\label{sec:proofs-of-main-theorems}

In the last section we show that the parametrix construction in lemma~\ref{lma:parametrix} in combination with the recent result~\cite[Theorem 1]{IlmKyk2021} can be used to prove our main results.

\begin{proof}[Proof of theorem~\ref{thm:elliptic-regularity}]
Let $f \in H^s_c(M)$ for some $s>-k+6+\frac{n}{2}$ and assume that $Nf = 0$. Let $\supp f \subseteq \Omega$. There is a cut-off function $\phi \in C^\infty_c(M)$ so that $\phi f = f$ and moreover there is a cut-off $\psi \in C^\infty_c(M)$ with $\psi = 1$ on $\Omega$ so that $Nf(x) = (\psi N \phi)f(x) = 0$ for all $x \in M$. The operator $\psi N\phi$ has Schwartz kernel of the form~\eqref{eqn:tilde-K} so by lemma~\ref{lma:parametrix} there is an operator $P$ and $\eps>0$ so that $P(\psi N \phi) = \Id + R$ where $\Id$ acts as the identity on elements in $H^t(\R^n)$ with support in $\Omega$ and $R \colon H^t_c(M) \to H^{t+\tau}(\R^n)$ is continuous for $\tau \in (0,\eps]$ and
\begin{equation}
-k+6+\frac{n}{2}-\tau
<
t
<
k-6-\frac{n}{2}-\tau.
\end{equation}
We may choose $\tau$ so small that $s > -k+6+\frac{n}{2}-\tau$. Then $f \in H^s_c(M)$ and
\begin{equation}
\phi f = P(\psi N\phi)f - Rf = -Rf.
\end{equation}
Thus $\phi f \in H^{t+\tau}(\R^n)$ and therefore $f \in H^{t+\tau}_c(M)$.

Then let $s < r < k-6-\frac{n}{2}$. By possibly choosing~$\tau$ to be even smaller we may assume that there is $m \in \N$ so that $r < s + m\tau < k-6-\frac{n}{2} -\tau$. Then by iterating $m$ times the argument in the previous paragraph we see that $f \in H^{s+m\tau}_c(M) \subseteq H^r_c(M)$ as claimed in the theorem.
\end{proof}

\begin{proof}[Proof of proposition~\ref{prop:on-l2}]
The composition of~$I$ and~$I^\ast$ was computed in~\cite[Lemma 8.1.5]{PSUbook} for $g \in C^\infty(M)$. The same computation works for $g \in C^k(M)$ when $k \ge 2$.
\end{proof}

\begin{proof}[Proof of theorem~\ref{thm:xrt-injective}]
Let $(\tilde M, \tilde g)$ be a simple extension of $(M,g)$ and let $\tilde I$ be the X-ray transform of $(\tilde M,\tilde g)$. Suppose that $f \in L^2(M)$ and $If = 0$. Then zero extension of $f$ to $\tilde M$ still denoted by $f$ satisfies $\tilde I f = 0$. Therefore $\tilde N f = \tilde I^\ast \tilde I f = 0$ by proposition~\ref{prop:on-l2} where $\tilde N$ and $\tilde I^\ast$ are the operators on $\tilde M$ defined by~\eqref{eqn:normal-operator} and~\eqref{eqn:back-projection} with all objects replaced by corresponding objects of $(\tilde M,\tilde g)$. Therefore by theorem~\ref{thm:elliptic-regularity} applied to the simple extension $(\tilde M,\tilde g)$ implies that $f \in H^r_c(\tilde M)$ for all $s < r < k-6+\frac{n}{2}$. Since $k \ge n + 8$ there is some $r \in \R$ so that $\lceil 1 + \frac{n}{2} \rceil < r < k-6+\frac{n}{2}$ and $f \in H^r_c(\tilde M)$. Sobolev embedding yields
\begin{equation}
H^r_c(\tilde M) \subseteq W^{1,\infty}(\tilde M) = \Lip(\tilde M).
\end{equation}
Thus $f \in \Lip(\tilde M)$ and since $f$ vanishes in $\tilde M \setminus M$ we have $f \in \Lip_0(M)$. We see that $f = 0$ since $I$ is injective on $\Lip(M)$ by~\cite[Theorem 1.]{IlmKyk2021} which finishes the proof.
\end{proof}

\bibliographystyle{plain}
\bibliography{references}

\end{document}